\documentclass[11 pt]{article}
\usepackage[margin=2.5cm]{geometry}
\usepackage{amsmath,amsthm,amsfonts,amssymb}
\usepackage{caption}
\usepackage{subcaption}

\usepackage{graphicx,latexsym}
\usepackage{epstopdf}
\usepackage{lscape}

\newtheorem{theorem}{Theorem}[section]

 \newtheorem{lemma}[theorem]{Lemma}

\newtheorem{remark}[theorem]{Remark}

\numberwithin{equation}{section}

\newtheorem{discu}{Discussion:}

\newtheorem{conje}{Conjecture:}

\begin{document}
\date{}

\title{Domination Parameters in Hypertrees and Sibling trees}

\author{\begin{tabular}{rcl}
         			Indra Rajasingh${^1}$, R. Jayagopal$^{2,}$\thanks{Corresponding Author: R. Jayagopal} ~and R. Sundara Rajan$^{3,}$\thanks{This work is supported by National Board of Higher Mathematics (NBHM), No. 2/48(4)/2016/NBHM-R$\&$D-II/11580, Department of Atomic Energy (DAE), Government of India.}
				\end{tabular}\\
        \begin{tabular}{c}
 					$^{1,2}$School of Advanced Sciences, Vellore Institute of Technology,\\ Chennai - 600 127, India\\
					$^{3}$Department of Mathematics, Hindustan Institute of Technology and Science,\\
Chennai - 603 103, India\\
				  $^{1}$indrarajasingh@yahoo.com, $^{2}$jgopal89@gmail.com and $^{3}$vprsundar@gmail.com \\
        \end{tabular}}

\maketitle

\begin{abstract}

A locating-dominating set (LDS) of a graph $G$ is a dominating set $S$ of $G$ such that for every two vertices $u$ and $v$ in $V(G) \setminus S$, $N(u)\cap S \neq N(v)\cap S$. The locating-domination number $\gamma^{L}(G)$ is the minimum cardinality of a LDS of $G$. Further if $S$ is a total dominating set then $S$ is called a locating-total dominating set. In this paper we determine the domination, total domination, locating-domination and locating-total domination numbers for hypertrees and sibling trees.\\


\end{abstract}
%
%
\section{Introduction}

A set $S$ of vertices in a graph $G$ is called a dominating set of $G$ if every vertex in $V(G) \setminus S$ is adjacent to some vertex in $S$.
The set $S$ is said to be a total dominating set of $G$ if every vertex in $V(G)$ is adjacent to some vertex in $S$.
The minimum cardinalities of a dominating set and a total dominating set of $G$ are denoted as $\gamma(G)$ and $\gamma_t(G)$, respectively.
Domination arises in facility location problems, where the number of facilities such as hospitals or fire stations are fixed and one attempts to minimize the distance that a person needs to travel to get to the closest facility.

Total domination plays a role in the problem of placing monitoring devices in a system in such a way that every site in the system, including the monitors, is adjacent to a monitor site so that, if a monitor goes down, then an adjacent monitor can still protect the system. Installing minimum number of expensive sensors in the system which will transmit a signal at the detection of faults and uniquely determining the location of the faults motivate the concept of locating-dominating sets and locating-total dominating sets \cite{How}.

In a parallel computer, the processors and interconnection networks are modeled by the graph $G=(V,E)$, where each processor is associated with a vertex of $G$ and a direct communication link between two processors is indicated by the existence of an edge between the associated vertices. Suppose we have limited resources such as disks, input-output connections, or software modules, and we want to place a minimum number of these resource units at the processors, so that every processor is adjacent to at least one resource unit, then finding such a placement involves constructing a minimum dominating set for the graph $G$.
Determining if an arbitrary graph has a dominating and locating-dominating sets of a given size are well-known $NP$-complete problems \cite{NP,Oli}. Occurrence of faulty nodes in a device is inevitable. So, to diagnose these faults we make use of locating-total domination set in this system. We place monitoring devices in a system in such a way that every site in the system (including the monitors) is adjacent to a monitor site.

A locating-dominating set $(LDS)$ in a connected graph $G = (V, E)$ is a dominating set $S$ of $G$ such that for every pair of vertices $u$ and $v$ in $V(G) \setminus S$, $N(u)\cap S \neq N(v)\cap S$. The minimum cardinality of a locating-dominating set of $G$ is called the locating-domination number $\gamma^{L}(G)$ {\rm \cite{How}}.
The locating-domination problem has been discussed for paths and cycles \cite{ld1,ld2}, infinite grids \cite{ld3}, circulant graphs \cite{ld4}, fault-tolerant graphs \cite{ld5} and so on.

A locating-total dominating set $(LTDS)$ in a connected graph $G = (V, E)$ is a total dominating set $S$ of $G$ such that for every pair of vertices $u$ and $v$ in $V(G) \setminus S$, $N(u)\cap S \neq N(v)\cap S$. The minimum cardinality of a locating total-dominating set of $G$ is called the locating-total domination number $\gamma^{L}_t(G)$ {\rm \cite{How}}.
The locating-total domination problem has been discussed for trees \cite{chena}, cubic graphs and grid graphs \cite{mic}, corona and composition of graphs \cite{b.n}, claw-free cubic  graphs \cite{m.a.c}, edge-critical graphs \cite{m.w} and so on.

Tree structures are expansible in a natural way, and even unbalanced trees still retain most of the properties that make the tree attractive.
Additional links, however, are required to reduce the average distance between nodes and to provide a more uniform message density in all links. An extensive search for the optimal placement of these additional links has shown the half-ring binary  trees such as hypertrees, sibling trees and christmas trees to be attractive contenders, primarily because of their simple routing algorithms.

In this paper, we determine the domination, total domination, locating-domination and locating-total domination numbers for hypertrees and sibling trees.
\begin{figure}[h!]
\centering
\includegraphics[scale=0.44]{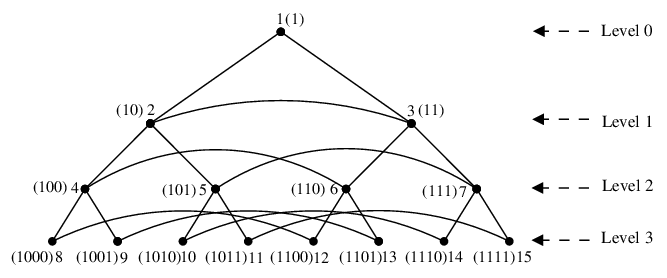}
\caption{$HT(3)$ with decimal labels and binary labels within braces}
\label{def}
\end{figure}
\section{Domination in hypertrees}
The basic skeleton of a hypertree is a complete binary tree $T_n$ of height $n$. Here the nodes of the tree are numbered as follows: The root node has label $1 $. The root is said to be at level $0$. Labels of left and right children are formed by appending a $0$ and $1$, respectively to the labels of the parent node. The decimal and binary labels of the hypertree are given in Figure {\rm \ref{def}}. Here the children of the node $x$ are labeled as $2x$ and $2x+1$. Additional links in a hypertree are horizontal and two nodes are joined in the same level $i$ of the tree if their label difference is $2^{i-1}$. We denote an $n$-level hypertree as $HT(n)$. It has $2^{n+1}-1$ vertices and $3(2^{n}-1)$ edges. Hypertree is a multiprocessor interconnection topology which has a frequent data exchange in algorithms such as sorting and Fast Fourier Transforms $(FFT's)$ \cite{goodman}.
The root-fault hypertree $HT^{*}(n)$, $n \geq 2$, is a graph obtained from $HT(n)$ by deleting the root vertex {\rm \cite{sundar}}. See Figure \ref{star}.
The following lemma is obvious from the definition of a hypertree.
\begin{lemma}
\label{lem}
The hypertree $HT(n)$, $n\geq3$, contains $2^{n-2}$ disjoint isomorphic copies of $HT^{*}(2)$ and $2^{n-3}$ disjoint isomorphic copies of $HT^{*}(3)$.
\end{lemma}
\begin{figure}[h!]
\vspace{-0.3cm}
\centering
\includegraphics[scale=0.45]{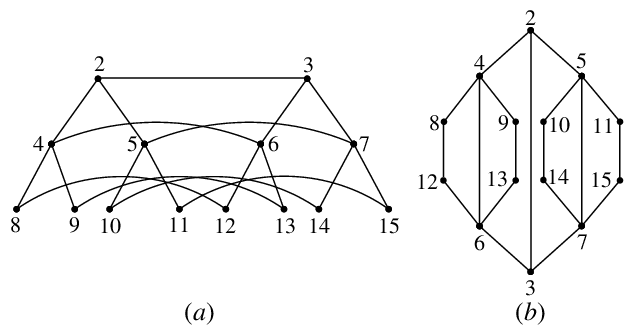}
\caption{ (a) $HT_{}^{*}(3)$ by definition~and~(b) $HT_{}^{*}(3)$ redrawn}
\label{star}
\vspace{-0.2cm}
\end{figure}
\begin{lemma}
\label{d_td}
Let $G$ be the root-fault hypertree $HT^{*}(2)$. Then $\gamma(G) =\gamma_t(G) = 2$.
\end{lemma}

\begin{proof}
Let $S$ be a dominating set of $G$.
We claim that $\left|S\right| \geq 2$.
Suppose not, let $\left|S\right| = 1$.
Then there exists a vertex $u$ in $S$ such that deg($u$) = $5$, a contradiction, since $\Delta(G) = 3$.
Hence $\left|S\right|\geq 2$.
Let $S=\{ v, v' \}$ where deg($v$) = deg($v'$) = $3$. See Figure \ref{abc}(a).
Now, $N[v] \cup N[v'] = V(G)$ and hence $\left|S\right| \leq  2$. Since $v$ and $v'$ are adjacent in $G$, $S$ is also a minimum total dominating set of $G$. Therefore $\gamma(G) =\gamma_t(G) = 2$.
\end{proof}

\begin{figure}[h!]
\centering
\includegraphics[scale=0.45]{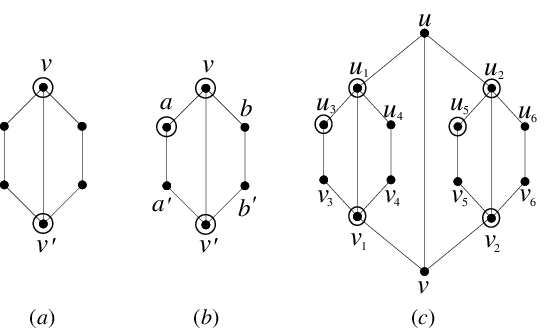}
\caption{Circled vertices constitute $(a)$ a minimum dominating set of $HT^{*}(2)$,~
(b)~a minimum locating-dominating set of $HT^{*}(2)$~and~(c)~a minimum locating-dominating set of $HT^{*}(3)$}
\label{abc}
\end{figure}

\begin{lemma}\label{3}
\label{ld_ltd_3}
Let $G$ be the root-fault hypertree $HT^{*}(2)$. Then $\gamma^{L}(G) =\gamma^{L}_t(G) = 3$.
\end{lemma}

\begin{proof}
Let $S$ be a locating-dominating set of $G$.
We claim that $\left|S\right| \geq 3$.
By Lemma \ref{d_td}, $\gamma^{L}(G) \geq 2$.
Assume that $\left|S\right| = 2$.
Let $S=\{ v, v' \}$ where deg($v$) = deg($v'$) = $3$. Then $N(v)=\{ a,b,v' \} $ and $N(v')=\{ a',b',v \} $. See Figure \ref{abc}(b). This implies $N(a) \cap S = \{v\} = N(b) \cap S$.
Suppose $S=\{a,b'\}$ then $N(v) \cap S = \{a\} =N(a') \cap S$.
Thus $\left|S\right| \geq 3$. Now let $S=\{v,v',a\}$. Then $N(a') \cap S=\{v',a\}, N(b) \cap S=\{v\}, N(b') \cap S=\{v'\}$ and $N[S]=V(G)$. Hence $\gamma^{L}(G) \leq 3$. Since vertices $v,v'$ and $a$ induce a path on 3 vertices in $G$, $S$ is also a minimum locating-total dominating set of $G$. Therefore  $\gamma^{L}(G) =\gamma^{L}_t(G) = 3$.
\end{proof}
\begin{lemma}
\label{last_level_d}
Let $G$ be the hypertree $HT(n), n \geq 1$.
Then any minimum dominating set of $G$ contains at least $2_{}^{n-1}$ vertices from levels $n-1$ and $n$.
\end{lemma}

\begin{proof}
Let $S$ be a minimum dominating set of $G$. Vertices in levels $n$ and $n-1$ of $G$ induce $2^{n-2}_{}$ copies of $H$, each isomorphic to $HT^{*}_{}(2)$. The worst case arises when both vertices of degree 3 in $H$ are already dominated by vertices from $G \setminus H$. By proof of Lemma $\ref{d_td}$, each copy contains at least 2 vertices of $H$. Hence $S$ contains at least $2(2^{n-2}_{})$ vertices from levels $n$ and $n-1$ in $G$.
\end{proof}

\begin{figure}[h!]
\centering
\includegraphics[scale=0.54]{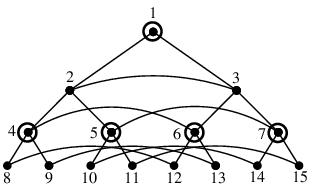}
\caption{Circled vertices constitute a minimum dominating set $S$ of {\it HT}($3$)}
\vspace{-0.5cm}
\label{td3(a) and d3(b)}
\end{figure}

\begin{theorem}
\label{dom}
Let $G$ be the hypertree $HT(n),n \geq 1$. Then \newline
$$\gamma(G) = \left\{
\begin{array}{c l}
    \frac{1}{7}(2^{n+2}+3) & if~~n\equiv0~(mod~3)\\\\
		\frac{1}{7}(2^{n+2}-1) & if~~n\equiv1~(mod~3)\\\\
		\frac{2}{7}(2^{n+1}-1) & if~~n\equiv2~(mod~3)
 \end{array}\right.$$
\label{dom}
\end{theorem}
\begin{proof}
We prove the result by induction on $n$.\\
{\bf Case $(i)$:  $n\equiv0~(mod~3)$}\\
Let $n= 3$ and let $S$ be a dominating set of $HT(3)$.
By Lemma $\ref{last_level_d}$, we need at least 4 vertices from levels 3 and 2 in $S$. To dominate the root vertex, we need at least one vertex from level 1 in $S$ or the root vertex itself has to be included in $S$. Therefore  $\left| S \right| \geq 5 = (1/7)(2^{3+2}+3)$.
				Now we will prove the equality.
				Let $S$ be the set of all vertices comprising of all vertices in level 2 and the root vertex of $HT(3)$. See Figure \ref{td3(a) and d3(b)}. Since all the vertices of level 3 and level 1 are adjacent to the vertices of level 2, $S$ is a dominating set of $HT(3)$.			 
				Therefore  $\left| S \right| \leq 5 = (1/7)(2^{3+2}+3)$.
				Assume that the result is true for $n=3k$, $k\geq1$.
That is, $\gamma(HT(3k))=(1/7)(2^{3k+2}+3).$ Consider $HT(3k+3)$. By Lemma \ref{lem}, there are $2^{3k+1}$ vertex disjoint copies of $HT^{*}(2)$ in $HT(3k+3)$.
Deletion of these subgraphs $HT^{*}(2)$ along with the vertices of $HT(3k+3)$ adjacent to vertices of these subgraphs results in $HT(3k)$.
Therefore by Lemma \ref{d_td},~$\gamma(HT(3k+3)) \leq \gamma(HT(3k))+2(2^{3k+1})$ and by induction hypothesis, $\gamma(HT(3k+3)) \leq (1/7)(2^{3k+2}+3)+2(2^{3k+1})=(1/7)(2^{(3k+3)+2}+3)$.
Now, let $S$ and $S_{1}^{}$ be the minimum dominating sets of $HT(3k+3)$ and $HT(3k)$, respectively.
					Let $S_{2}^{} \subset S$ be the vertex set which contains the vertices from the last three levels of $HT(3k+3)$.
				  Similar to the argument for $n=3$, any minimum dominating set contains at least $2^{3k+2}_{}$ vertices from levels $3k+2$ and $3k+3$.
		Therefore,~$\gamma(HT(3k+3)) \geq \left| S_{1}^{} \right| + \left| S_{2}^{} \right|
					= \big[ (1/7)(2^{3k+2}+3) \big] + \big[ 2^{3k+2} \big] = (1/7)(2^{(3k+3)+2}+3)$.\\
The case when $n\equiv 1,2~(mod~3)$ can be dealt with similarly.
\end{proof}		

\begin{figure}[h!]
\centering
\includegraphics[scale=0.41]{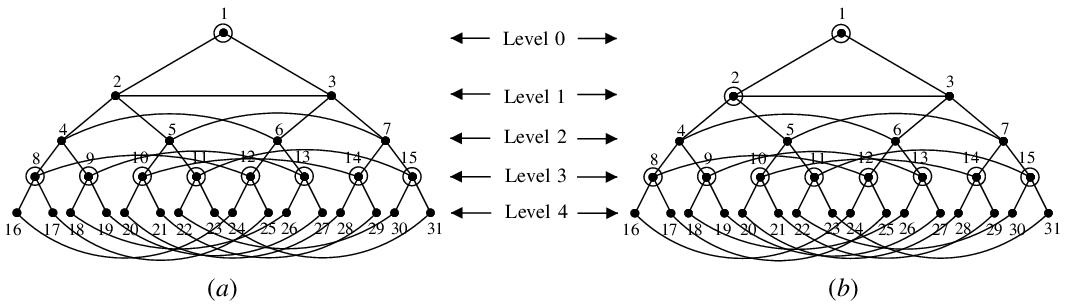}
\caption{Circled vertices constitute $(a)$ a minimum dominating set of {\it HT}($4$)~and~($b$)~a minimum total dominating set of {\it HT}($4$)}
\label{d4}
\end{figure}
\begin{remark}
The dominating sets described in Theorem $\ref{dom}$ for $HT(n)$, when $n\equiv0, 2~(mod~3)$ do not contain any isolated vertex.
Let $n = 3k+1$, $k \geq 0$. By Lemma $\ref{last_level_d}$, recursively every set of three levels from the bottom of $HT(n)$ must contain a minimum number of vertices from the last two levels to dominate all the three levels. Hence no minimum dominating set of $HT(n)$ contains any vertex from level $2$. So far, the vertices of level 0 and level 1 are not dominated. Hence any minimum dominating set contains at least one vertex either from level 0 or level 1.
See Figure \ref{d4}$(a)$.
Now, to make it as a total dominating set, any minimum total dominating set of $HT(n)$ must include one more vertex either from level $0$ or level $1$. See Figure \ref{d4}$(b)$.
These observations yield the following result.
\end{remark}

\begin{theorem}
Let $G$ be the hypertree $HT(n), n \geq 1$. Then \newline
\hspace{12cm} $$\gamma_t(G) = \left\{
\begin{array}{c l}
    \frac{1}{7}(2^{n+2}+3) & if~~n\equiv0~(mod~3)\\\\
		\frac{1}{7}(2^{n+2}-1)+1 & if~~n\equiv1~(mod~3)\\\\
		\frac{2}{7}(2^{n+1}-1) & if~~n\equiv2~(mod~3)
		\end{array}\right.$$
\end{theorem}

\begin{lemma}
\label{ld_ltd_6}
Let $G$ be the root-fault hypertree $HT^{*}(3)$. Then $\gamma^{L}(G) =\gamma^{L}_t(G) = 6$.
\end{lemma}
\begin{proof}

Let $S$ be a locating-dominating set of $G$. Assume that $\left|S\right| \leq 5$.
The vertices $u$ and $v$ are the only two vertices of degree 3 in $G$.
We assume that $u$ and $v$ do not belong to $S$. It is easy to see that the removal of $u$ and $v$ disconnects $G$ into two components $G_1$ and $G_2$ which are isomorphic to $HT^{*}(2)$. See Figure \ref{abc}(c). We need at least 3 vertices each to identify all the vertices in $G_{1}$ and $G_{2}$. This contradicts the cardinality of $S$.
Suppose $u$ and $v$ belongs to $S$, then we need at least 2 vertices in each of $G_{1}$ and $G_{2}$ to dominate $G_{1}$ and $G_{2}$.
This again contradicts the cardinality of $S$.
The case when either $u$ or $v$ belongs to $S$ is similar. Therefore $\gamma^{L}(G)\geq 6$.
Label the vertices of $G$ as in Figure \ref{abc}(c) and let
$S= \{ u_{1}, u_{2}, u_{3}, u_{5}, v_{1},  v_{2} \}$.
It is easy to check that $S$ is a locating-dominating set of $G$. Further there are no isolated vertices in the subgraph induced by $S$. Therefore $S$ is also a locating-total dominating set of $G$.
Hence $\gamma^{L}(G) =\gamma^{L}_t(G)=6.$
\end{proof}
\begin{remark}
Let $S$ be a dominating set of a graph $G$. A pair of vertices $u$ and $v$ of $V(G) \setminus S$ is said to be located by $S$ if $N(u)\cap S \neq N(v)\cap S$. We also say that $S$ locates $u$ and $v$. If $S$ is a locating-dominating set, then $S$ locates every pair of vertices in $V(G) \setminus S$.
\end{remark}

\begin{lemma}
\label{last_level_ld}
Let $G$ be the hypertree $HT(n), n \geq 1$.
Then any minimum locating-dominating set of $G$ contains at least $2_{}^{n-1}$ vertices of $G$ from level $n$.
\end{lemma}

\begin{proof}
Let $S$ be a minimum locating-dominating set of $G$.
In $G$, the vertices of level $n$ induce a perfect matching consisting of $k=2^{n-1}_{}$ copies of complete graphs, $K_2$, say $H_{1}^{}, H_{2}^{}, \ldots , H_{k}^{}$.
Suppose $xy$ be an edge in $H_{1}^{}$ and $\{x, y\} \cap S = \emptyset $, then $N(x) \cap S = N(y) \cap S$.
Therefore $S$ contains at least one vertex from each of $V(H_{1}^{}), V(H_{2}^{}), \ldots , V(H_{k}^{})$.
\end{proof}

\begin{figure}[h!]
\centering
\includegraphics[scale=0.44]{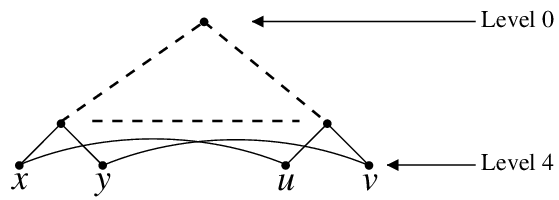}
\caption{Illustrates the proof of Case $(i)$ in Theorem \ref{ht ld}}
\vspace{-0.5cm}
\label{xyuv}
\end{figure}

\begin{theorem}
Let $G$ be the hypertree $HT(n),n \geq 1$. Then
 $$\gamma^{L}(G) = \left\{
\begin{array}{c l}
    \frac{1}{5}(2^{n+2}+1) & if~~n\equiv0~(mod~4)\\\\
		\frac{1}{5}(2^{n+2}+2) & if~~n\equiv1~(mod~4)\\\\
		\frac{1}{5}(2^{n+2}-1) & if~~n\equiv2~(mod~4)\\\\
		\frac{1}{5}(2^{n+2}-2) & if~~n\equiv3~(mod~4)
 \end{array}\right.$$
\label{ht ld}
\end{theorem}			

\begin{proof}
We prove the result by induction on $n$.\\
\noindent{\bf Case $(i)$ :  $n\equiv0~(mod~4)$}\\
Let $n=4$ and $S$ be a minimum locating-dominating set of $HT(4)$.
				Let $xu$ and $yv$ be the edges in $HT(4)$ with vertices $x,u,y,v$ in level 4 such that $a$ is the parent of $x$ and $y$, and $b$ is the parent of $u$ and $v$. See Figure \ref{xyuv}.
					By Lemma $\ref{last_level_ld}$, we need at least 8 vertices from levels 3 in $S$.
					If $\{ x, y \} \subset S$, then we need at least one vertex either from level 3 or level 2 to dominate the vertex $b$. If $\{ x, v \} \subset S$, then to locate the vertices $a$ and $u$ we need at least one vertex either from level 3 or from level 2. In either case, for four vertices in level 4, at least one vertex from level 3 or level 2 gets included in $S$. Since there are 16 vertices in level 4, at least 4 more vertices from level 3 and level 2 get included in $S$. However, to dominate the root vertex we need at least one vertex from level 1 in $S$ or the root vertex itself has to be included in $S$. Therefore  $\left| S \right| \geq 13$.
										Now we will prove the equality.
					Let $S$ be the set of all vertices in level 0 and level 2 together with four alternate vertices beginning from left to right and the another 4 alternate vertices beginning from right to left in level 4.	
					See Figure \ref{ht ld4_ltd2}(b).
					Now, each vertex of level $1$ are located by its children in level 2.
					Let $x$ and $y$ be two vertices in level 3 which belongs to $V(HT(4)) \setminus S$.
					If $x$ and $y$ has different parent in level 2, then $N(x) \cap S \neq N(y) \cap S$.
					If $x$ and $y$ has same parent in level 2, then $N(x) \cap S \neq N(y) \cap S$, since at least one child of $x$ and $y$ in level 4 are in $S$. Since $S$ contains one vertex from every edge of level 4, $S$ is a locating-dominating set of $HT(4)$.
					Therefore $\left| S \right| \leq 13$. Thus, $S$ is a minimum locating-dominating set of $HT(4)$ and hence $\gamma^{L}(HT(4))=13=(1/5)(2^{4+2}+1).$
						Assume that the result is true for $n=4k$, $k\geq1$. That is, $\gamma^{L}(HT(4k))=(1/5)(2^{4k+2}+1)$.
						Consider $HT(4k+4)$. By Lemma \ref{lem}, there are $2^{4k+1}$ vertex disjoint copies of $HT^{*}(3)$ in $HT(4k+4)$. Deletion of these subgraphs $HT^{*}(3)$ along with the vertices of $HT(4k+4)$ adjacent to vertices of these subgraphs, results in $HT(4k)$.
						Therefore by Lemma \ref{ld_ltd_6},~$\gamma^{L}(HT(4k+4)) \leq \gamma^{L}(HT(4k))+6(2^{4k+1})$ and by induction hypothesis, $\gamma^{L}(HT(4k+4)) \leq (1/5)(2^{4k+2}+1)+6(2^{4k+1})=(1/5)(2^{(4k+4)+2}+1)$.
Now, let $S$ and $S_{1}^{}$ be the minimum locating-dominating set of $HT(4k+4)$ and $HT(4k)$, respectively.
Let $S_{2}^{} \subset S$ be the vertex set which contains the vertices from the last four levels of $HT(4k+4)$.	
Similar to the argument for $n=4$, any minimum locating-dominating set contains at least $2^{4k+3}_{}$ vertices from level $n$ and at least $2^{4k+2}_{}$ vertices from level $4k+3$ and level $4k+2$.
Therefore, $\gamma^{L}(HT(4k+4)) \geq \left| S_{1}^{} \right| + \left| S_{2}^{} \right|
		 = \big[ (1/5)(2^{4k+2}+1) \big] + \big[ 2^{4k+3}_{} + 2^{4k+2}_{} \big]=(1/5)(2^{(4k+2)+4}+1)$.\\
						%
The cases when $n\equiv 1,2,3~(mod~4)$ can be dealt with similarly.
\end{proof}

\begin{figure}[h!]
\centering
\includegraphics[scale=0.45]{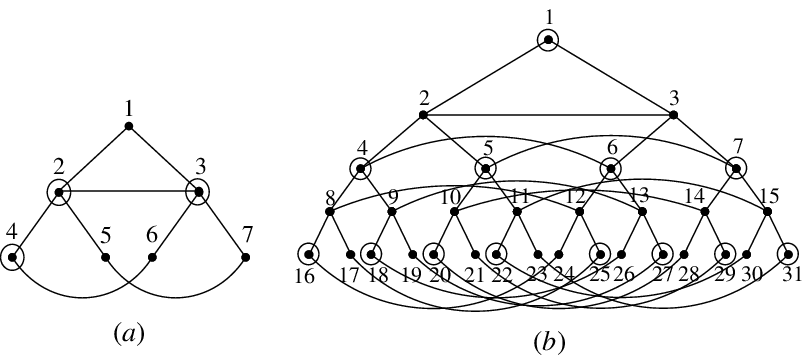}
\caption{Circled vertices constitute $(a)$ a minimum locating-total dominating set of {\it HT}($2$)~and
$(b)$~a minimum locating-dominating set of {\it HT}($4$)}
\label{ht ld4_ltd2}
\end{figure}
\begin{figure}[h!]
\centering
\includegraphics[scale=0.45]{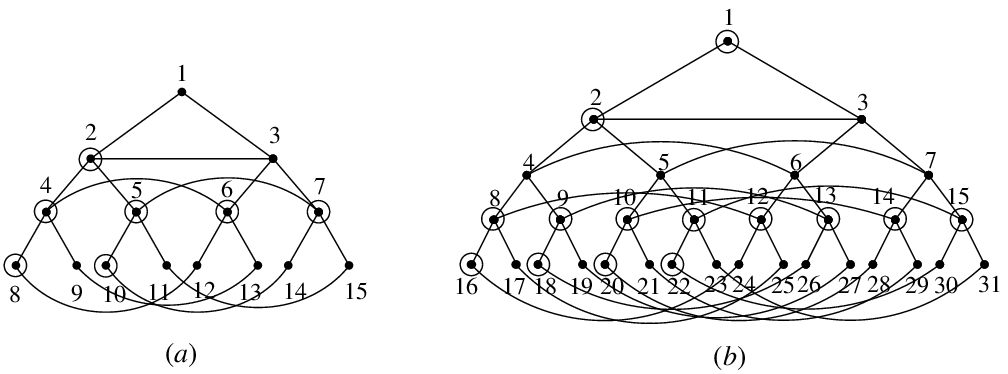}
\caption{Circled vertices constitute a minimum locating-total dominating set of $(a)$ {\it HT}($3$) and $(b)$ {\it HT}($4$)}
\label{ltd3_4}
\end{figure}

\begin{lemma}
\label{last_level_ltd}
Let $G$ be the hypertree $HT(n), n \geq 1$.
Any minimum locating-total dominating set contains at least $3(2^{n-2}_{})$ vertices from levels $n-1$ and $n$ in $G$.
\end{lemma}

\begin{proof}
Let $S$ be a minimum locating-total dominating set of $G$. Vertices in levels $n$ and $n-1$ of $G$ induce $2^{n-2}_{}$ copies of $H$, each isomorphic to $HT^{*}_{}(2)$. The worst case arises when both vertices of degree 3 in $H$ are already located by vertices from $G \setminus H$. By proof of Lemma $\ref{ld_ltd_3}$, each copy contains at least 3 vertices of $H$. Hence $S$ contains at least $3(2^{n-2}_{})$ vertices from levels $n$ and $n-1$ in $G$.
\end{proof}

\begin{theorem}
Let $G$ be the hypertree $HT(n),n \geq 1$. Then
 $$\gamma^{L}_t(G) = \left\{
\small \begin{array}{c l}
   \frac{1}{7}(3(2^{n+1})+1) & if~n\equiv0~(mod~3)\\\\
		\frac{2}{7}(3(2^{n})+1) & if~n\equiv1~(mod~3)\\\\
		\frac{3}{7}(2^{n+1}-1) & if~n\equiv2~(mod~3)
 \end{array}\right.$$
\end{theorem}			
\begin{proof}
We prove the result by induction on $n$.\\
{\bf Case $(i)$:  $n\equiv0~(mod~3)$}\\
Let $n= 3$ and let $S$ be a locating-total dominating set of $HT(3)$.
By Lemma $\ref{last_level_ltd}$, we need at least 6 vertices from levels 2 and 3 in $S$. To dominate the root vertex, we need at least 1 vertex from level 1. Thus $\gamma^{L}_t(HT(3)) \geq 7 $. 			
				Let $S$ be the set of all vertices comprising of all vertices in level 2, one of the vertex in level 1 and two alternate vertices beginning from left to right in level 3. See Figure \ref{ltd3_4}(a).
				Now, each vertex of level $2$ are located by its children in level 3.
				Let $x$ and $y$ be two vertices in level 4 which belongs to $V(HT(3)) \setminus S$.
				Let $u$ and $v$ be two vertices in level 4 such that $u \in N(x)$ and $v \in N(y)$.
If $x$ and $y$ has same parent, say $w$, in level $3$ then $N(x) \cap S \neq N(y) \cap S$ since either $u \in  S$ or $v \in S$.
If $x$ and $y$ has different parent, say $a$ and $b$, respectively. In this case there are two chances either $u$ and $v$ not belonging to $S$ or any one of them belonging to $S$.
			  If $\{ u , v \} \not\subset S$, then $N(x) \cap S \neq N(y) \cap S$ since $N(x) \cap S = \{ a \}$ and $N(y) \cap S = \{ b \}$.
				If $u \in  S$ and $v \notin S$, then $N(x) \cap S \neq N(y) \cap S$ since $N(x) \cap S = \{ a,u \}$ and $N(y) \cap S = \{ b \}$. Also it is easy to see that the vertices in level 1 are located by $S$. Thus $S$ is a locating-dominating set of $HT(4)$.	Therefore $\gamma^{L}_t(HT(3)) \leq 7 = (1/7)(3(2^{3+1})+1)$.
									Assume that the result is true for $n=3k$, $k\geq1$. That is, $\gamma^{L}_t(HT(3k))=(1/7)(3(2^{3k+1})+1).$ Consider $HT(3k+3)$. By Lemma \ref{lem}, there are $2^{3k+1}$ vertex disjoint copies of $HT^{*}(2)$ in $HT(3k+3)$. Deletion of these subgraphs $HT^{*}(2)$ along with the vertices of $HT(3k+3)$ adjacent to vertices of these subgraphs results in $HT(3k)$. Therefore by Lemma \ref{ld_ltd_3},~$\gamma^{L}_t(HT(3k+3)) \leq \gamma^{L}_t(HT(3k))+3(2^{3k+1})$ and by induction hypothesis,
$\gamma^{L}_t(HT(3k+3)) \leq (1/7)(3(2^{3k+1})+1) + 6(2^{3k})=(1/7)(3(2^{(3k+3)+1})+1)$.
Now, let $S$ and $S_{1}^{}$ be the minimum locating-total dominating set of $HT(3k+3)$ and $HT(3k)$, respectively.
Let $S_{2}^{} \subset S$ be the vertex set which contains the vertices from the last three levels of $HT(3k+3)$.	
Similar to the argument for $n=3$, any minimum locating-total dominating set contains at least $3(2^{3k+1}_{})$ vertices from levels $3k+3$ and $3k+2$.
Therefore, $\gamma^{L}_t(HT(3k+3)) \geq \left| S_{1}^{} \right| + \left| S_{2}^{} \right|
		 = \big[ (1/7)(3(2^{3k+1})+1) \big] + \big[ 2^{3k+2}_{} + 2^{3k+1}_{} \big] = (1/7)(3(2^{(3k+3)+1})+1)$.\\
%
The case when $n\equiv 1~(mod~4)$ can be dealt with similarly.
For illustration, the locating-total dominating set of $HT(4)$ is given in Figure \ref{ltd3_4}(b).\\
\noindent The case when $n \equiv 2~(mod~3)$ is similar with $S=\{ 2, 3, 4 \}$ being the minimum locating-total dominating set of $HT(2)$ as the base case. See Figure \ref{ht ld4_ltd2}(a).
\end{proof}

\section{Domination in Sibling trees}

Sibling tree is obtained from the complete binary tree $T_{n}$ by adding edges (sibling edges) between left and right children of the same parent node.
Here the nodes of the sibling tree are numbered as follows: The root node has label $1$. The root is said to be at level $0$.
Here the children of the nodes $x$ are labeled as $2x$ and $2x+1$. See Figure \ref{st def}.
We denote an $n$-level sibling tree as $ST_{n}$. It has $2^{n+1}-1$ vertices and $3(2^{n}-1)$ edges.
%
For each $i, 1 \leq i \leq n$, let $V_{i}^{}$ denote the vertex set in level $i$, with $\left| V_{i}^{} \right| =2^i$. We call the edges in level $n$ as terminal edges and the vertices incident on them as terminal vertices.
The following lemma is obvious from the definition of a sibling tree.
%

%

\begin{figure}[h!]
\centering
\includegraphics[scale=0.41]{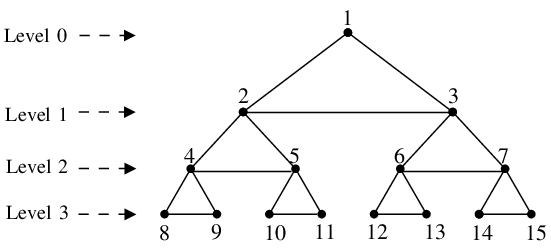}
\caption{$ST_{3}$ with decimal label}
\label{st def}
\end{figure}

\begin{lemma}\label{3}
\label{d_td_st2}
Let $G$ be the sibling tree $ST_{2}$. Then $\gamma^{}(G) =\gamma^{}_t(G) = 2$.
\end{lemma}

\begin{proof}
Let $S$ be a dominating set of $G$.
We claim that $\left|S\right| \geq 2$.
Suppose not, let $\left|S\right| = 1$.
Then there exists a vertex $u$ in $S$ such that deg($u$) = $6$, a contradiction since $\Delta(G)=4$. Hence $\left|S\right|\geq 2$.
Let $S = V_{1}^{}$. Then $N[S] = V(G)$ and hence $\left|S\right| \leq  2$.
Since the vertices in $V_{1}^{}$ are adjacent in $G$, $S$ is also a minimum total dominating set of $G$.
Therefore $\gamma(G) =\gamma_t(G) = 2$.
\end{proof}

The proof of the following lemma is similar to that of Lemma \ref{last_level_d} and hence is omitted.\\

\begin{lemma}
\label{last_level_st_d}
Let $G$ be the sibling tree $ST_{n}, n \geq 1$.
Then any minimum dominating set of $G$ contains at least $2_{}^{n-1}$ vertices from levels $n-1$ and $n$.
\end{lemma}

Using Lemma \ref{d_td_st2} and Lemma \ref{last_level_st_d}, we will prove the following theorem and the proof is similar to that of Theorem \ref{dom} and hence is omitted.\\

\begin{theorem}
\label{st domination}
Let $G$ be the sibling tree $ST_{n}, n \geq 1$. Then \newline
$$\gamma(G) = \left\{
\begin{array}{c l}
    \frac{1}{7}(2^{n+2}+3) & if~~n\equiv0~(mod~3)\\\\
		\frac{1}{7}(2^{n+2}-1) & if~~n\equiv1~(mod~3)\\\\
		\frac{2}{7}(2^{n+1}-1) & if~~n\equiv2~(mod~3)
 \end{array}\right.$$
\end{theorem}

\begin{remark}
The dominating sets described in Theorem $\ref{st domination}$ for $ST_{n}$, when $n\equiv0, 2~(mod~3)$ do not contain any isolated vertex.
Let $n = 3k+1$, $k \geq 0$. By Lemma $\ref{last_level_st_d}$, recursively every set of three levels from the bottom of $ST_{n}$ must contain a minimum number of vertices from the last two levels to dominate all the three levels. Hence no minimum dominating set of $ST_{n}^{}$ contains any vertex from level $2$.
So far, the vertices of level 0 and level 1 are not dominated. Hence any minimum dominating set contains at least one vertex either from level 0 or level 1. See Figure $\ref{st d td}(a)$.
Now, to make it as a total dominating set, any minimum total dominating set of $ST_{n}^{}$ must include one more vertex either from level $0$ or level $1$. See Figure $\ref{st d td}(b)$.
These observations yield the following result.
\end{remark}

\begin{theorem}
\label{st dom}
Let $G$ be the sibling tree $ST_{n}, n \geq 1$. Then \newline
\hspace{12cm} $$\gamma_t(G) = \left\{
\begin{array}{c l}
    \frac{1}{7}(2^{n+2}+3) & if~~n\equiv0~(mod~3)\\\\
		\frac{1}{7}(2^{n+2}-1)+1 & if~~n\equiv1~(mod~3)\\\\
		\frac{2}{7}(2^{n+1}-1) & if~~n\equiv2~(mod~3)
		\end{array}\right.$$
\label{st dom}
\end{theorem}

\begin{figure}[h!]
\centering
\includegraphics[scale=0.55]{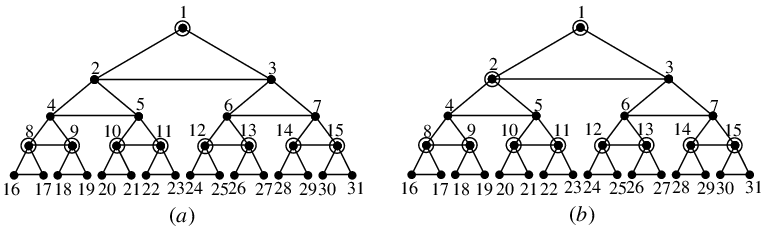}
\caption{Circled vertices constitute $(a)$ a minimum dominating set of ${\it ST}_{4}$~and
$(b)$~a minimum total dominating set of ${\it ST}_{4}$}
\label{st d td}
\end{figure}

\begin{lemma}\label{3}
\label{st_ld}
Let $G$ be the sibling tree $ST_{3}$. Then $\gamma^{L}(G) = 6$.

\end{lemma}

\begin{proof}
Let $S$ be a locating-dominating set of $G$.
We claim $\left|S\right| \geq 6$.
From each terminal edge $uv$, at least one vertex of it should belongs to $S$,
otherwise $N(u)\cap S = N(v)\cap S$, a contradiction. Thus $\left| S \right| \geq 4$.
The terminal vertices do not dominate the vertices in $V_{0}$ and $V_{1}.$ Hence $\left| S \right| \geq 5$.
Suppose $\left| S \right| = 5$. Let $w$ be the non terminal vertex in $S$. If $w$ is in $V_{2}$, then the root vertex is not dominated. On the other hand, if $w$ is in level 0 or 1, then the other two vertices $x$ and $y$ in $V_{0} \cup V_{1}$ are such that $N(x)\cap S = N(y)\cap S$, a contradiction.
Hence  $\left| S \right| \geq 6$.
\end{proof}

The proof of the following lemma is similar to that of Lemma \ref{last_level_ld} and hence is omitted.\\

\begin{lemma}
\label{last_level_st_ld}
Let $G$ be the sibling tree $ST_{n}, n \geq 1$.
Then any minimum locating-dominating set of $G$ contains at least $2_{}^{n-1}$ vertices of $G$ from level $n$.
\end{lemma}

Using Lemma \ref{st_ld} and Lemma \ref{last_level_st_ld}, we will prove the following theorem and the proof is similar to that of Theorem \ref{ht ld} and hence is omitted.\\


\begin{theorem}
Let $G$ be the sibling tree $ST_{n},n \geq 1$. Then
 $$\gamma^{L}(G) = \left\{
\begin{array}{c l}
    \frac{1}{5}(2^{n+2}+1) & if~~n\equiv0~(mod~4)\\\\
		\frac{1}{5}(2^{n+2}+2) & if~~n\equiv1~(mod~4)\\\\
		\frac{1}{5}(2^{n+2}-1) & if~~n\equiv2~(mod~4)\\\\
		\frac{1}{5}(2^{n+2}-2) & if~~n\equiv3~(mod~4)
 \end{array}\right.$$
\end{theorem}	

\begin{figure}[h!]
\centering
\includegraphics[scale=0.55]{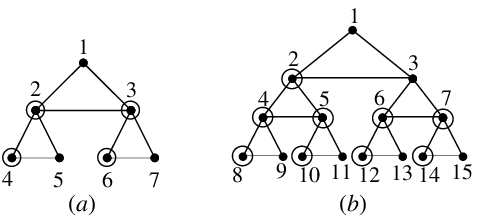}
\caption{Circled vertices constitute $(a)$ a minimum locating-total dominating set of ${\it ST}_{2}$ and
$(b)$~a minimum locating-total dominating set of ${\it ST}_{3}$}
\label{st ltd ld}
\end{figure}

\begin{lemma}\label{3}
\label{ltd_st2}
Let $G$ be the sibling tree $ST_{2}$. Then $\gamma^{L}_{t}(G) = 4$.
\end{lemma}

\begin{proof}
Let $S$ be a locating-total dominating set of $G$.
We see that at least one vertex from each terminal edge should belong to $S$ in order to locate the vertices in level $n$ distinctly.
Let $u$ and $v$ be the selected vertices , one each from the two terminal edges.
Since $u$ and $v$ are isolated vertices, to obtain a total domination, we need at least two vertices $x$ and $y$ such that $\{ ux, vy \} \in E(G)$.
Hence $\gamma_{t}(ST_{2}) \geq 4$.
Let $S = V_{1} \cup \{u, v\}$. Then clearly $S$ is a locating-total dominating set of $G$. See Figure \ref{st ltd ld}(a).
Therefore $\gamma^{L}_{t}(G) = 4$.
\end{proof}

\begin{lemma}
\label{last_level_sbt_ltd}
Let $G$ be the sibling tree $ST_{n}, n \geq 1$.
Any minimum locating-total dominating set of $G$ contains at least $2^{n}_{}$ vertices from levels $n-1$ and $n$.
\end{lemma}

\begin{proof}
Let $S$ be a minimum locating-total dominating set of $G$. The vertices in levels $n$ and $n-1$ induce a subgraph $H$ consisting of $2^{n}_{}$ copies of complete graph $K_{3}^{}$. The worst case arises when the vertices in level $n-1$ are already located by vertices in $G \setminus H$. However, every $K_{3}^{}$ should contain 2 vertices of $S$.
\end{proof}

\begin{theorem}
Let $G$ be the sibling tree $ST_{n},n \geq 1$. Then
 $$\gamma^{L}_{t}(G) = \left\{
\begin{array}{c l}
    \frac{1}{7}(2^{n+3}-1) & if~~n\equiv0~(mod~3)\\\\
		\frac{1}{7}(2^{n+3}-2) & if~~n\equiv1~(mod~3)\\\\
		\frac{1}{7}(2^{n+3}-4) & if~~n\equiv2~(mod~3)
 \end{array}\right.$$
\end{theorem}	

\begin{proof}
We prove the result by induction on $n$.\\
{\bf Case $(i)$:  $n\equiv0~(mod~3)$}\\
Let $n= 3$ and let $S$ be a locating-total dominating set of $ST_{3}$.
By Lemma \ref{last_level_sbt_ltd}, we need 8 vertices from levels 3 and 2 in $S$. To dominate the root vertex, we need at least 1 vertex from level 1. Thus $\left| S \right| \geq 9 $. 	
Now we will prove the equality.		
				Let $S$ be the set of all vertices comprising of all vertices in level 2 and the alternate vertices beginning from left to right in level 3. See Figure $\ref{st ltd ld}(b)$.
				Let $x$ and $y$ be two vertices in level 3 which belongs to $V(HT(3)) \setminus S$.
If $x$ and $y$ has same parent, then $N(x) \cap S \neq N(y) \cap S$ since either $x \in  S$ or $y \in S$.
If $x$ and $y$ has different parent, then $N(x) \cap S \neq N(y) \cap S$ since one vertex from every edge in level 3
belongs to $S$. Also it is easy to see that the vertices in level 0 and level 1 are located by $S$. Thus $S$ is a locating-dominating set of $HT(4)$.	Therefore $\left| S \right|  \leq 9 = (2^{3+3}-1)/7$.
Assume that the result is true for $n=3k$, $k\geq1$.
Consider $ST_{3k+3}$.
Deletion of vertices in levels $3k+1, 3k+2$ and $3k+3$ in $ST_{3k+3}$ yields $ST_{3k}$.
By induction hypothesis, $\gamma^{L}_{t}(ST_{3k})= (1/7)(2^{3k+3}-1)$. There are $2^{3k+1}$ vertex disjoint copies of $ST_{2}$ in the subgraph induced by vertices in the levels $3k+1, 3k+2$ and $3k+3$ of $ST_{3k+3}.$
Therefore by Lemma \ref{ltd_st2}, $\gamma^{L}_{t}(ST_{3k+3}) \leq (1/7)(2^{3k+3}-1)+4(2^{3k+1})=(1/7)(2^{(3k+3)+3}-1)$.
Now, let $S$ and $S_{1}^{}$ be the minimum locating-total dominating set of $ST_{3k+3}$ and $ST_{3k}$, respectively.
Let $S_{2}^{} \subset S$ be the vertex set which contains the vertices from the last three levels of $ST_{3k+3}$. Similar to the argument for $n=3$, any minimum locating-total dominating set contains at least $2^{3k+3}_{}$ vertices from levels $3k+3$ and $3k+2$.
Therefore $\gamma^{L}_{t}(ST_{3k+3}) \geq \left| S_{1}^{} \right| + \left| S_{2}^{} \right|
		 = \big[ (1/7)(2^{3k+3}-1) \big] + \big[ 2^{3k+3}_{} \big] = (1/7)(2^{(3k+3)+3}-1)$.\\
The case when $n\equiv1,2~(mod~3)$ can be dealt with similarly.
\end{proof}

\section{Conclusion}
In this paper, we have proved that $\gamma(G)=\gamma_t(G)$  when $G$ is a hypertree $HT(n)$, $n\equiv0,2~(mod~3)$ and $\gamma(G)=\gamma_t(G)-1$ when $G$ is $HT(n), n\equiv1~(mod~3)$. We have also computed $\gamma^{L}(HT(n))$ and $\gamma^{L}_t(HT(n))$, $n \geq1$. We have obtained similar results for sibling tree $ST_{n}, n\geq1$.
Finding classes of graphs $G$ with $\gamma(G)=\gamma_t(G)=\gamma^{L}(G)=\gamma^{L}_t(G)$ is under investigation.
%

%

\end{document}